\documentclass[12pt,reqno]{amsart}

\usepackage{amssymb,latexsym}

\usepackage{enumerate}
\allowdisplaybreaks
\usepackage[french,english]{babel}
\usepackage{amsmath}
\usepackage{graphicx}
\usepackage{amssymb}
\usepackage{bbm}
\usepackage{amsthm,mathtools}
\usepackage{ulem}
\usepackage{geometry}
\usepackage{tikz-cd}
\usepackage{mathrsfs}
\usepackage[colorinlistoftodos]{todonotes}
\usepackage{enumitem}
\usepackage{verbatim}
\usepackage[foot]{amsaddr}
\usepackage{dsfont}
\usepackage{cite}
\usepackage[T1]{fontenc}

\makeatletter

\@namedef{subjclassname@2010}{
	
	\textup{2020} Mathematics Subject Classification}

\makeatother
\newtheorem{thm}{Theorem}[section]
\newtheorem*{thm*}{Theorem}

\newtheorem{cor}{Corollary}[section]%Corollary单独排序，不和其他混淆

\newtheorem{lem}[thm]{Lemma}

\theoremstyle{definition}

\newtheorem{rem}{Remark}[section]%Remark单独排序，不和其他混淆
\numberwithin{equation}{section}

\newcommand{\inv}{^{-1}}

\newcommand{\mbz}{\mathbb{Z}}

\newcommand{\mbn}{\mathbb{N}}

\newcommand{\mce}{\mathcal{E}}

\newcommand{\mcg}{\mathcal{G}}

\newcommand{\mcp}{\mathcal{P}}

\newcommand{\mme}{\mathrm{e}}
\newcommand{\mmi}{\mathrm{i}}

\newcommand{\re}{\textup{Re}}
\newcommand{\im}{\textup{Im}}

\usepackage{hyperref}
\hypersetup{colorlinks=true,linkcolor=blue,anchorcolor=blue,citecolor=blue}
\usepackage{color}
\newcommand{\newabstract}[1]{%
	\par\bigskip
	\csname otherlanguage*\endcsname{#1}%
	\csname captions#1\endcsname
	\item[\hskip\labelsep\scshape\abstractname.]
}

\begin{document}

	\baselineskip=17pt

	\title[Joint extreme values of Dirichlet \(L\)-functions and their logarithmic derivatives]{Joint extreme values of Dirichlet \(L\)-functions and their logarithmic derivatives}

    \author{Shengbo Zhao\textsuperscript{1}}
	\address{1.School of Mathematical Sciences, Key Laboratory of Intelligent Computing and Applications(Ministry of Education), Tongji University, Shanghai 200092, P. R. China}	\email{shengbozhao@hotmail.com}

	\begin{abstract} 
    In this paper, we establish joint extreme values of Dirichlet \(L\)-functions and their logarithmic derivatives using the resonance method. Our results extend previous work of Aistleitner et al. (2019) and Yang (2023).
	\end{abstract}
	
	
	
	\maketitle

\section{Introduction}
Dirichlet \(L\)-functions \(L(s,\chi)\) form a fundamental class of objects in analytic number theory, playing a central role in the study of various arithmetic, geometric, and algebraic problems, where \(s = \sigma + \mmi t\). Understanding the behavior of their values for \(\sigma \in[1/2,1]\) is a central problem. The study of Dirichlet \(L\)-functions has a long history. Over the past decades, the introduction of the notion of families of \(L\)-functions has led to significant progress. In particular, modeling such families by characteristic polynomials of random matrices has provided powerful heuristics and yielded celebrated breakthroughs in the field. 
\par
Extreme values of \(L\)-functions reflect the distribution of their values and are closely connected to problems concerning character sums and class numbers. When \(\sigma=1\), by using high moments of \(L\)-functions and results on sums involving the divisor function, Granville and Soundararajan \cite{Granville2006Ramanujan} refined Littlewood’s earlier work and showed that, for sufficiently large primes \(q\), there exist at least \(q^{1-1/A}\) characters \(\chi \pmod{q}\) such that
\[
|L(1,\chi)| \ge \mme^\gamma(\log_2 q+\log_3 q-\log_4 q-\log A+O(1))
\]
for any \(A \ge 10\). Here, \(\gamma\) denotes the Euler–Mascheroni constant. Throughout this paper, we assume that \(q\) is a sufficiently large prime, \(\chi\) is a character modulo \(q\), and \(\log_k\) denotes the \(k\)-th iterated logarithm. The currently best known result was established by Aistleitner, Mahatab, Munsch, and Peyort \cite{Aistleitner2019QJMath}. Using the resonance method, they removed the term \(-\log_4 q\) inside the parentheses and proved there exists a non-principal character \(\chi\) such that
\begin{align}\label{AMMP1line}
    |L(1,\chi)| \ge \mme^\gamma(\log_2 q+\log_3 q-C+o(1)),
\end{align}
where \(C=1+\log_2 4 \approx 1.33\). Notably, their result matches the order predicted by probabilistic models. Both \cite{Aistleitner2019QJMath} and \cite{Granville2006Ramanujan} also provide quantitative information on the frequency of such characets \(\chi\) for which \(|L(1,\chi)|\) attains extreme values.
\begin{rem}
    In \cite{Aistleitner2019QJMath}, \(q\) is merely taken to be sufficiently large. In fact, the assumption that \(q\) is prime is essential, and counterexamples can be found in the work of D. Yang; see \cite[Remark 3]{yang2023BLMS}. In the subsequent proofs, we will repeatedly use the assumption that \(q\) is prime.
\end{rem}
When \(\sigma \in (1/2,1)\), Lamzouri \cite{lamzouri2011IMRN} established precise results on the distribution of extreme values for several families of \(L\)-functions. For the family of \(L\)-functions associated with Legendre symbols \(\chi_p=\big(\frac{\cdot}{p} \big) \), Lamzouri showed, under the generalized Riemann hypothesis (GRH), that for sufficiently large \(x\), there are \(\gg x^{1/2}\) primes \(p \le x\) such that 
\[
\log L(\sigma+\mmi t,\chi_p) \ge (\beta(s)+o(1))(\log x)^{1-\sigma}(\log_2 x)^{-\sigma}.
\]
When \(t=0\), this refines earlier results of Montgomery \cite{Montgomery1977} on the Riemann zeta function in the critical strip, under the Riemann hypothesis (RH); see \cite[Remark 3]{lamzouri2011IMRN}. Further related results can be found in \cite{Granville2003Geom}. Moreover, Aistleitner et al. \cite{Aistleitner2019QJMath} proved that there exists a non-principal character \(\chi\) such that 
\begin{align}\label{AMMPcriticalstrip}
    \log |L(\sigma,\chi)| \ge C(\sigma)(\log q)^{1-\sigma} (\log_2 q)^{-\sigma}
\end{align}
for some constant \(C(\sigma)>0\), which is conjectured to be optimal up to a constant. This result was later refined by X. Xiao and Q. Yang \cite{XiaoYang2022BAMS}, who provided an explicit estimate for \(C(\sigma)\).
\par

The study of extreme values of logarithmic derivatives of \(L\)-functions is also of independent interest. Their values at the point \(s=1\) are related to the Euler–Kronecker constants of global fields, with cyclotomic fields being a prominent case. Moreover, as one of the classical examples of \(L^\prime/L(s,\chi)\), logarithmic derivatives of the Riemann zeta function are closely related to the distribution of primes, as it appears in the proof of the prime number theorem. It has also been conjectured by D. Yang \cite[Conjecture 10.1]{yang2023omega} that extreme values of \(L^\prime/L(s,\chi)\) may be related to the ratio between extreme values of \(L\)-functions and their derivatives \(L^\prime(s,\chi)\).
\par
There are few results on extreme values of logarithmic derivatives of \(L\)-functions. Mourtada and Murty \cite{Mourtada2013IJNT} showed that there exist infinitely many fundamental discriminants \(D\) such that 
\[
\frac{L^\prime}{L}(1,\chi_D) \ge \log_2 |D|+O(1),
\]
where \(\chi_D\) denotes the quadratic Dirichlet character of conductor \(D\). Moreover, under GRH, they also proved that there are \(\gg x^{1/2}\) primes \(p \le x\) such that
\[
\frac{L^\prime}{L}(1,\chi_p) \ge \log_2 x +\log_3 x +O(1)
\]
for sufficiently large \(x\). 
\par
D. Yang \cite{yang2023omega} studied the case for \(\sigma\in(1/2,1]\) and proved that there exists a non-principal character \(\chi\) such that 
\begin{align}\label{Yangomgea1line}
    -\re \frac{L^\prime}{L}(1,\chi) \ge \log_2 q +\log_3 q+ C_1 +o(1),
\end{align}
and for \(\sigma \in (1/2,1)\),
\begin{align}\label{Yangomgeacriticalstrip}
    -\re \frac{L^\prime}{L}(\sigma,\chi) \ge C_2(\sigma)(\log q)^{1-\sigma}(\log_2 q)^{-\sigma},
\end{align}
where \(C_1 \) and \(C_2(\sigma)\) are positive constants that can be computed effectively. For all primes \(q \ge 10^{10}\), he also showed that there exists a non-principal character \(\chi\) such that 
\begin{align}\label{Yangomegatheta}
    \re\Big(\mme^{-\mmi\theta}\frac{L^\prime}{L}(1,\chi)\Big)\ge \log_2 q +O(\log_3 q).
\end{align}
This result captures extreme values in different directions. Compared with \eqref{Yangomegatheta}, the case \(\theta=\pi/2\) makes \eqref{Yangomgea1line} more explicit. Meanwhile, he established corresponding results for the Riemann zeta function and quantitative estimates on the frequency of extreme values when \(\sigma=1\).
\par
In this paper, we continue the study of Q. Yang and S. Zhao \cite{qiyuyang2026joint}, which was motivated by the work of Levinson \cite{levinson1972harmonic}. In \cite{qiyuyang2026joint}, joint extreme values of the Riemann zeta function at harmonic points were studied both on the 1-line and in the critical strip. Henceforth, we fix an integer \(\ell \in \mbz^+\). We study joint extreme values over powers of Dirichlet characters \(\chi\), namely
\[
\chi,\chi^2,\dots,\chi^\ell.
\]
Our first theorem establishes joint extreme values of \(L\)-functions at the point \(s=1\).

\begin{thm}\label{thm1}
    Let \(\ell \ge 1\) be a fixed integer. For all sufficiently large primes \(q\), there exists a Dirichlet character \(\chi \, (\operatorname{mod} q)\) with \(\operatorname{ord}(\chi) > \ell\) such that
    \[
\prod_{j=1}^\ell \big|L(1,\chi^j)\big| \ge \mme^{\ell\gamma}\big\{(\log_2 q)^\ell + \ell(\log_3 q - C(\ell))(\log_2 q)^{\ell-1} + O\big((\log_2 q)^{\ell-2}(\log_3 q)^2\big) \big\},
    \]
    where \(C(\ell) = (\ell+1)/2 + \log_2 4\). The implied constant in the \(O(\cdot)\) only depends on \(\ell\).
\end{thm}
\begin{rem}
    The condition \(\operatorname{ord}(\chi) > \ell\) ensures that, for all \(j \in \{1,\dots,\ell\}\), none of \(\chi^j\) coincides with the principal character \(\chi_0\). When \(\ell=1\), Theorem \ref{thm1} recovers \eqref{AMMP1line} of Aistleitner et al. \cite{Aistleitner2019QJMath}. Moreover, the secondary term in Theorem \ref{thm1} reflects the interaction among different powers, showing that extreme values are slightly compressed as \(\ell\) grows. Our error term is also sharper.
\end{rem}
\par
The following theorem studies joint extreme values of \(L\)-functions when \(\sigma \in (1/2,1)\).
\begin{thm}\label{thm2}
    Let \(\ell \ge 1\) be a fixed integer and let \(1/2 < \sigma <1\). For all sufficiently large primes \(q\), there exists a Dirichlet character \(\chi \, (\operatorname{mod} q)\) with \(\operatorname{ord}(\chi) > \ell\) such that
    \[
    \prod_{j=1}^\ell \big|L(\sigma,\chi^j)\big| \ge \exp \bigg\{\big(\kappa^{1-\sigma}S(\sigma,\ell)+o(1)\big)\frac{(\log q)^{1-\sigma}}{(\log_2 q)^\sigma} \bigg\}.
    \]
    Here \(S(\sigma,\ell)\) is an explicit constant defined by
    \[
    S(\sigma,\ell) = \frac{\ell}{1-\sigma} + \sum_{m=1}^\ell (-1)^m \binom{\ell+1}{m+1} \frac{1}{1+\sigma(m-1)},
    \]
    and \(\kappa = \kappa(\sigma)>0\) is a computable constant. The implied constant in the \(o(\cdot)\) only depends on \(\sigma\).
\end{thm}
\begin{rem}
    When \(\ell=1\), Theorem \ref{thm2} recovers \eqref{AMMPcriticalstrip} of Aistleitner et al. \cite{Aistleitner2019QJMath}. The leading term is more explicit. When \(\sigma\) is close to \(1\), \(S(\sigma,\ell)\) increases significantly with \(\ell\), whereas when \(\sigma\) is close to \(1/2\), its growth is suppressed. This is due to the strong cancellation produced by the alternating binomial sum involving the term \(\ell/(1-\sigma)\). 
\end{rem}
\begin{rem}
    In earlier work of Q. Yang and S. Zhao \cite{qiyuyang2026joint}, joint extreme values of the Riemann zeta function at harmonic points were established, namely at the points
    \[
    s=\sigma+\mmi t, \sigma+2\mmi t, \dots, \sigma+\ell\mmi t.
    \]
    In this paper, we focus on joint extreme values over powers of a sequence of Dirichlet characters \(\chi \pmod{q}\), that is, \(\chi,\chi^2,\dots,\chi^\ell\). In fact, both Theorems \ref{thm1} and \ref{thm2} in this paper are natural analogues, in the setting of Dirichlet \(L\)-functions, of joint extreme values problem for the Riemann zeta function at harmonic points in \cite{qiyuyang2026joint}.
\end{rem}
\par
Next, we turn to joint extreme values of logarithmic derivatives of \(L\)-functions. At the point \(s=1\), we obtain the following result.
\begin{thm}\label{thm3}
    Let \(\ell \ge 1\) be a fixed integer. For all sufficiently large primes \(q\), there exists a Dirichlet character \(\chi \, (\operatorname{mod} q)\) with \(\operatorname{ord}(\chi) > \ell\) such that
    \[
    (-1)^\ell \operatorname{Re} \prod_{j=1}^\ell \frac{L^\prime}{L}(1,\chi^j)\ge (\log_2 q)^\ell + \big(\ell\log_3 q+ Q(\ell)\big)(\log_2 q)^{\ell-1}+O((\log_2 q)^{\ell-2}(\log_3 q)^2).
    \]
    Here \(Q(\ell)\) is an explicit constant defined by
    \[
    Q(\ell) = \ell\Big(1-\log_2 4-\gamma - \sum_{p}\frac{\log p }{p(p-1)}\Big)-(\ell+1)\sum_{k=1}^\ell \frac{1}{k}.
    \]
    The implied constant in the \(O(\cdot)\) only depends on \(\ell\).
\end{thm}
\begin{rem}
    In \(Q(\ell)\), the coefficient of the first term on the right-side hand is approximately \(-0.659\). In fact, \(Q(\ell)\) is always negative and decreases rapidly as \(\ell\) increases. More precisely, \(Q(\ell) \sim -\ell\log\ell\) as \(\ell \to \infty\).
\end{rem}
\begin{rem}
    Compared with D. Yang’s result \eqref{Yangomgea1line} in \cite{yang2023omega}, Theorem \ref{thm3} provides a more explicit secondary term and a sharper error term. When \(\ell=1\), Theorem \ref{thm3} recovers \eqref{Yangomgea1line}.
\end{rem}
From Theorem \ref{thm3}, we can directly obtain the following Corollary \ref{corollary4}.

\begin{cor}\label{corollary4}
    Let \(\ell \ge 1\) be a fixed integer. For all sufficiently large primes \(q\), there exists a Dirichlet character \(\chi \, (\operatorname{mod} q)\) with \(\operatorname{ord}(\chi) > \ell\) such that
    \[
   \prod_{j=1}^\ell \bigg|\frac{L^\prime}{L}(1,\chi^j)\bigg| \ge (\log_2 q)^\ell + (\ell\log_3 q+ Q(\ell))(\log_2 q)^{\ell-1}+O((\log_2 q)^{\ell-2}(\log_3 q)^2).
    \]
    The implied constant in the \(O(\cdot)\) only depends on \(\ell\).
\end{cor}
\par
Our following theorem studies joint extreme values of \(L^\prime/L(s,\chi)\) when \(\sigma \in (1/2,1)\).
\begin{thm}\label{thm4}
    Let \(1 \le \ell <(2-2\sigma)\inv \) be a fixed integer. For all sufficiently large primes \(q\), there exists a Dirichlet character \(\chi \pmod{q}\) with \(\operatorname{ord}(\chi) > \ell\) such that
    \[
    (-1)^\ell \operatorname{Re} \prod_{j=1}^\ell \frac{L^\prime}{L}(\sigma,\chi^j)\ge \big(\eta^{\ell(1-\sigma)}H(\sigma,\ell)+o(1) \big)(\log q)^{\ell(1-\sigma)}(\log_2 q)^{\ell(1-\sigma)}.
    \]
    Here \(H(\sigma,\ell)\) is an explicit constant defined by
    \[
    H(\sigma,\ell) = \prod_{j=1}^\ell \bigg(\frac{j!}{1-\sigma}\prod_{m=0}^{j-1} \Big(m+\frac{1}{\sigma}\Big)\inv\bigg),
    \]
    and \(\eta = \eta(\sigma)>0\) is a computable constant. The implied constant in the \(o(\cdot)\) only depends on \(\sigma\).
\end{thm}
\begin{rem}
    In Theorem \ref{thm4}, \(\ell\) is bounded by \((2-2\sigma)\inv\), which imposes a rather restrictive condition on \(\ell\), especially when \(\sigma\) is close to \(1/2\). When \(\sigma\) approaches \(1\), the situation becomes more favorable. When \(\sigma\) is close to \(1\), \(H(\sigma,\ell)\) increases significantly with \(\ell\). More precisely, \(H(\sigma,\ell) = (1-\sigma)^{-\ell+o(\ell)}\) as \( \sigma \to 1^-\). As \(\sigma\) moves away from \(1\), the growth of \(H(\sigma,\ell)\)  becomes slower, since it is constrained by the upper bound for \(\ell\).
\end{rem}

\begin{rem}
    Theorem \ref{thm4} recovers \eqref{Yangomgeacriticalstrip} when \(\ell=1\), while also providing a more explicit secondary term and a sharper error term. For \(\sigma \in (1/2,1)\), assuming GRH, the ranges of \(\kappa\) and \(\eta\) in Theorems \ref{thm2} and \ref{thm4} can be slightly refined.
\end{rem}

Similarly, by Theorem \ref{thm4}, we can directly obtain the following Corollary \ref{corollary6}.
\begin{cor}\label{corollary6}
    Let \(1 \le \ell <(2-2\sigma)\inv \) be a fixed integer. For all sufficiently large primes \(q\) and some positive number \(\eta\), there exists a Dirichlet character \(\chi \, (\operatorname{mod} q)\) with \(\operatorname{ord}(\chi) > \ell\) such that
    \[
    \prod_{j=1}^\ell \bigg|\frac{L^\prime}{L}(\sigma,\chi^j)\bigg|\ge \big(\eta^{\ell(1-\sigma)}H(\sigma,\ell)+o(1) \big)(\log q)^{\ell(1-\sigma)}(\log_2 q)^{\ell(1-\sigma)}.
    \]
    The implied constant in the \(o(\cdot)\) only depends on \(\sigma\).
\end{cor}
\par
The study of joint extreme values of \(L\)-functions suggests that Dirichlet characters attaining extreme values are not independent, but are constrained by an underlying multiplicative structure. In other words, such extreme values tend to occur simultaneously within the power family generated by a single character, indicating a strong correlation. 
\par
From an algebraic perspective, Dirichlet characters modulo \(q\) form a multiplicative group, and the sequence \(\chi,\chi^2,\dots,\chi^\ell\) corresponds to the power map \(\chi \mapsto \chi^j\) on this group. Studying joint extreme values of \(L(s,\chi^j)\) can be viewed as investigating of the distribution of values of \(L\)-functions under this group action. In this sense, such results not only describe extreme values, but also reveal structural properties of the character group.
\par
Moreover, logarithmic derivatives of \(L\)-functions are closely related to the distribution of their zeros, and extreme values reflect clustering of zeros in certain regions. Studying joint extreme values of \(L^\prime/L(s,\chi)\) therefore reveals a form of joint behavior in the zero distribution of \(L\)-functions. More precisely, for characters \(\chi\) attaining extreme values, several functions \(L(s,\chi^j)\) may simultaneously approach zeros near the same point.
\par
In this paper, we mainly employ the resonance method, which can be traced back to the work of Voronin \cite{voronin1988lower} and was later refined by Soundararajan \cite{soundararajan2008extreme}. Aistleitner \cite{aistleitner2016MathAnn} introduced the long resonator method by observing the role of greatest common divisor sums. Subsequently, Bondarenko and Seip \cite{bondarenko2017Duke,bondarenko2018MathAnn} further developed this approach and improved results on extreme values of the Riemann zeta function and its argument. For more details and results about the resonance method, we recommend \cite{aistleitner2019IMRN,bondarenko2023dichotomy,chirre2019extreme,dong2023Onde,xumax2024JNT,qiyu2024JNT} and the references therein.
\par
We now introduce some notation and briefly outline the structure of this paper. Let \(\mcg_q\) denote the set of Dirichlet characters modulo \(q\), and let \(\chi_0\) denote the principal character. Let \(A\) and \(\varepsilon\) be arbitrarily positive numbers, where each occurrence may represent a different value. Let \(\mbn\) denote the set of natural numbers, and let \(p\) denote a prime. We write \((m,n)\) for the greatest common divisor of \(m\) and \(n\), and denote by \(\phi(n)\) and \(\Lambda(n)\) the Euler's totient function and the von Mangoldt function, respectively. In Section \ref{auxiliarylemmas}, we present several lemmas, some of which rely on the residue theorem and classical results of Gronwall, Landau, and Titchmarsh (see \cite[Principle 1]{heath1992zerofree}). More precisely, there exists a constant \(A >0\) such that, in the region
\[
\Big\{s=\sigma+\mmi t: \sigma > 1- \frac{A}{\log(q(|t|+2))} \Big\},
\] 
there is at most one exceptional character \(\chi_\mme \pmod{q}\) for which \(L(s,\chi_\mme)=0\); for all other characters \(\chi \neq \chi_\mme\), we have \(L(s,\chi) \neq 0\). Accordingly, we set \(\mcg_q^\ast = \mcg_q \setminus\{\chi_0,\chi_\mme\}\). Furthermore, to study joint extreme values, we define
\[
\mcg_\ell(q) = \{\chi\,(\operatorname{mod} q): \chi^j \neq \chi_0,\chi_\mme \ \text{for all}\ 1 \le j \le \ell \}.
\]
\par
In Sections \ref{sectionforthm1} and \ref{sectionforthm2}, we will prove Theorems \ref{thm1} and \ref{thm2}. The proofs of Theorems \ref{thm3} and \ref{thm4} are given in Sections \ref{sectionforthm3} and \ref{sectionforthm4}.

\section{Auxiliary Lemmas}\label{auxiliarylemmas}
In this section, we introduce several lemmas that will be used in the subsequent proofs. Define \(L(1,\chi;Y)=\prod_{p \le Y}\big(1-\chi(p)p\inv\big)\inv\). The following lemma, which provides a good approximation to \(L(1,\chi)\), will be used in the proof of Theorem \ref{thm1}.

\begin{lem}[Aistleitner-Mahatab-Munsch-Peyort \cite{Aistleitner2019QJMath}]
\label{lemma1forthm1}
    Let \(Y = \exp\big((\log q)^{20}\big)\). Then we have
    \[
    L(1,\chi) = L(1,\chi;Y)\big(1+O\big((\log q)^{-2}\big)\big), \ \ \forall \chi \in \mcg_q^\ast.
    \]
\end{lem}
\begin{proof}
    It follows directly from \cite[Eq. (2.1)]{Aistleitner2019QJMath}.
\end{proof}

\par
The next lemma, established by Granville and Soundararajan \cite{Granville2001JAMS}, provides an
important estimate for \(\log L(s,\chi)\).
\begin{lem}[Granville-Soundararajan \cite{Granville2001JAMS}]
\label{lemma1forthm2}
Let \(s=\sigma+\mmi t\) with \(\sigma>1/2\) and \(|t|\ge 2q\). Ley \(y \ge 2\) be a real number, and let \(1/2 \le \sigma_0 <\sigma\). Suppose there are no zeros of \(L(z,\chi)\) inside the rectangle \(\{z:\sigma_0\le \re(z) \le1,|\im(z)-t|\le y+3\}\). Put \(\sigma_1 = \min((\sigma+\sigma_0)/2,\sigma_0+1/\log y)\). Then
\[
\log L(s,\chi) = \sum_{n=2}^y \frac{\Lambda(n)\chi(n)}{n^s\log n}+O\Big(\frac{\log q}{(\sigma_1-\sigma_0)^2}y^{\sigma_1-\sigma} \Big).
\]
\end{lem}
\begin{proof}
    It follows directly from \cite[Lemma 8.2]{Granville2001JAMS}.
\end{proof}

\par
Note that Lemma \ref{lemma1forthm2} relies on a zero-free region assumption, therefore, zero-density results for Dirichlet \(L\)-functions are essential. Let \(N(\sigma,T,\chi)\) denote the number of zeros of \(L(s,\chi)\) inside \(\{z: \re(z) \ge \sigma,|\im(z)| \ge T\}\). The following zero-density estimate holds.
\begin{lem}
    \label{lemma2forthm2and4}
    For \(1/2 \le \sigma_0 \le 1\), \(T \ge 2\), we have
    \[
    \sum_{\chi \in \mcg_q}N(\sigma_0,T,\chi) \ll (qT)^{\frac{3(1-\sigma_0)}{2-\sigma_0}}(\log qT)^{14}.
    \]
\end{lem}
\begin{proof}
    This result is given in \cite[Theorem 12.1]{Montgomery1971}.
\end{proof}
\par
Combining Lemmas \ref{lemma1forthm2} and \ref{lemma2forthm2and4}, we obtain the following result, which provides an approximate formula for \(L(\sigma,\chi)\). This lemma will be used in the proof of Theorem \ref{thm2}.
\begin{lem}
    \label{lemma3forthm2}
    Let \(Y=(\log q)^{\frac{3}{\sigma-1/2}}\). Then
    \[
    L(\sigma,\chi)=L(\sigma,\chi;Y)\big(1+O\big((\log q)^{-\frac{1}{4}}\big)\big), \ \ \forall \chi \in \mcg_q\setminus(\mce(q)\cup\{\chi_0\}),
    \]
    where the cardinality of \(\mce(q)\) satisfies
    \[
    \#\mce(q) \ll q^{\frac{9/4-3\sigma/2}{7/4-\sigma/2}+o(1)}.
    \]
\end{lem}
\begin{proof}
    Applying Lemmas \ref{lemma1forthm2} and \ref{lemma2forthm2and4} with \(y=Y=(\log q)^{\frac{3}{\sigma-1/2}}\), \(\sigma_0=\sigma/2+1/4\), \(T=Y+2\) and \(t=0\), we obtain
    \[
    \log L(\sigma,\chi) = \sum_{n=2}^Y \frac{\Lambda(n)\chi(n)}{n^\sigma\log n}+O\big((\log q)^{-\frac{1}{4}}\big), \ \ \forall \chi \in \mcg_q\setminus(\mce(q)\cup\{\chi_0\}).
    \]
    Here, the cardinality of the set of characters \(\mce(q)\) satisfies
    \[
    \#\mce(q) \ll q^{\frac{9/4-3\sigma/2}{7/4-\sigma/2}+o(1)}.
    \]
    By the definition of \(\Lambda(n)\), we write \(n=p^k\) with \(k \ge 1\):
    \[
    \sum_{n=2}^Y \frac{\Lambda(n)\chi(n)}{n^\sigma\log n} = \sum_{p \le Y}\frac{\chi(p)}{p^\sigma}+\sum_{k \ge 2}\sum_{p^k \le Y}\frac{\chi(p)^k}{kp^{k\sigma}}.
    \]
    Furthermore, 
    \[
    L(\sigma,\chi;Y) = \exp\Big(-\sum_{p \le Y}\log\Big(1-\frac{\chi(p)}{p^\sigma}\Big) \Big) = \exp\Big(\sum_{p \le Y}\frac{\chi(p)}{p^\sigma}+ \sum_{k \ge 2}\sum_{p \le Y}\frac{\chi(p)^k}{kp^{k\sigma}}\Big).
    \]
    Since 
    \[
    \sum_{k \ge 2}\sum_{Y^{1/k}< p \le Y} \frac{\chi(p)^k}{kp^{k\sigma}} = O\big( (\log q)^{-\frac{1}{4}}\big),
    \]
    the proof follows.
\end{proof}
\par
The following lemma provides an approximate formula for \(L^\prime/L(1,\chi)\), and will be used in the proof of Theorem \ref{thm3}.
\begin{lem}[D. Yang \cite{yang2023omega}]
\label{lemma1forthm3}
    Let \(Y=\exp\big((3\log q)^2 \big)\). Then for some \(A >0\), we have
    \[
    -\frac{L^\prime}{L}(1,\chi) = \sum_{n \le Y}\frac{\Lambda(n)\chi(n)}{n}+O(q^{-A}), \ \ \forall \chi \in \mcg_q^\ast.
    \]
\end{lem}
\begin{proof}
    This result can be found in \cite[pp. 8-9]{yang2023omega}.
\end{proof}
\par
Similar to Lemma \ref{lemma1forthm2}, the following lemma shows that, within a certain zero-free region, \(L^\prime/L(s,\chi)\) can be approximated by a Dirichlet polynomial.
\begin{lem}[D. Yang \cite{yang2023omega}]
\label{lemma1forthm4}
    Let \(q\) be a prime greater than \(3\). Let \(Y \ge 3\), \(-3q \le t \le 3q\), and \(1/2 \le \sigma_0 < 1\). Suppose that the rectangle \(\{z: \sigma_0 < \re(z) \le 1,|\im(z)-t|\le Y+2\}\) is free of zeros of \(L(z,\chi)\), where \(\chi \in \mcg_q\). Then for any \(\sigma \in (\sigma_0,3]\) and \(\xi \in [t-Y,t+Y]\), we have
    \[
    \bigg|\frac{L^\prime}{L}(\sigma+\mmi\xi,\chi) \bigg| \ll \frac{\log q}{\sigma-\sigma_0}.
    \]
    Further, for \(\sigma \in (\sigma_0,1]\) and \(\sigma_1 \in (\sigma_0,\sigma)\), we have
    \[
    -\frac{L^\prime}{L}(\sigma+\mmi t,\chi) = \sum_{n \le Y}\frac{\Lambda(n)\chi(n)}{n^{\sigma+\mmi t}} +O\Big(\frac{\log q}{\sigma_1 - \sigma_0}Y^{\sigma_1-\sigma}\log\frac{Y}{\sigma-\sigma_1}\Big).
    \]
\end{lem}
\begin{proof}
    This result can be found in \cite[Lemma 2]{yang2023omega}.
\end{proof}
\par
Combining Lemmas \ref{lemma2forthm2and4} and \ref{lemma1forthm4}, we obtain the following result, which will be used in the proof of Theorem \ref{thm4}.
\begin{lem}
\label{lemma2forthm4}
Let \(1 \le \ell <(2-2\sigma)\inv \) be a fixed integer and let \(\sigma \in (1/2,1)\). Let \(\omega\in((1-\sigma)(\ell-1),\sigma-1/2)\) be fixed and let \(\varepsilon \in (0,\sigma-1/2)\) be small. Let \(Y=(\log q)^\beta\), where
\[
\beta > \frac{1}{\omega-(1-\sigma)(\ell-1)}>1.
\]
Then for some \(A>0\), we have
\[
-\frac{L^\prime}{L}(\sigma,\chi) = \sum_{n \le Y}\frac{\Lambda(n)\chi(n)}{n^\sigma}+O\big((\log q)^{-A}\big), \ \ \forall \chi \in \mcg^\ast_q \setminus \mce(q).
\]
Here, the cardinality of \(\mce(q)\) satisfies
\[
\#\mce(q) \ll q^{\frac{3(1-\sigma+\omega+\varepsilon)}{2-\sigma+\omega+\varepsilon}+o(1)}.
\]
\end{lem}
\begin{proof}
    In Lemma \ref{lemma1forthm4}, we set \(\sigma_0 = \sigma-\omega-\varepsilon\), \(\sigma_1 = \sigma-\omega\) and \(t=0\). Then, combining this with Lemma \ref{lemma2forthm2and4}, the proof is complete.
\end{proof}

\section{Proof of Theorem \ref{thm1}}\label{sectionforthm1}
In this section, we apply the long resonator method from \cite{Aistleitner2019QJMath} to prove Theorem \ref{thm1}. To this end, we first establish an effective approximate formula for \(L(1,\chi^j)\). Set \(Y = \exp\big((\log q)^{20}\big)\). By Lemma \ref{lemma1forthm1}, for all \(j \in \{1,\dots,\ell\}\), we have
\[
L(1,\chi^j) = L(1,\chi^j;Y)\big(1+O\big((\log q)^{-2}\big)\big), \ \ \forall \chi^j \in \mcg_q^\ast.
\]
Taking the product over \(1 \le j \le \ell\), we obtain
\begin{align}\label{approximateforthm1}
    \prod_{j=1}^\ell L(1,\chi^j) = \prod_{j=1}^\ell L(1,\chi^j;Y)\big(1+O\big((\log q)^{-2}\big)\big), \ \ \forall \chi \in \mcg_\ell(q).
\end{align}
\par
Set \(X=\log q \log_2 q /\delta\), where \(\delta>0\) will be chosen later. Note that \(X \le Y\). Following \cite{Aistleitner2019QJMath} (see also \cite{yang2023omega}), we define the resonator
\[
R(\chi) = \prod_{p \le X}(1-r(p)\chi(p))\inv = \sum_{n \in \mbn}r(n)\chi(n).
\]
Here \(r(n)\) is a completely multiplicative function, whose values at primes 
\(p\) satisfy
\begin{align*}
    r(p) =
    \begin{cases}
        1-\frac{p}{X}, ~ &\operatorname{if} p \le X, \\
		0, ~ &\operatorname{if} p > X.
    \end{cases}
\end{align*}
It remains to establish extreme values of \(\prod_{j=1}^\ell L(1,\chi^j;Y)\). For this purpose, we define the following two sums:
\[
S_1 = \sum_{\chi \in \mcg_q} |R(\chi)|^2 \ \text{and} \ \ S_2 = \sum_{\chi \in \mcg_q}\prod_{j=1}^\ell L(1,\chi^j;Y)|R(\chi)|^2.
\]
\par
For \(S_1\), the orthogonality of characters gives
\begin{align}\label{S1equationforthm1}
    S_1 = \sum_{\chi \in \mcg_q} \sum_{m,n \in \mbn}r(m)r(n)\chi(m)\overline{\chi(n)} = \phi(q)\sum_{\substack{m,n \in \mbn \\ m\equiv n\pmod{q}\\(n,q)=1}}r(m)r(n).
\end{align}
According to the definition of \(r(n)\), we can write \(L(1,\chi^j;Y)=\sum_{k_j \ge 1}b_{k_j}\chi(k_j)^j\). Here, \(b_{k_j}=1/k_j\) if all prime factors of \(k_j\) are not exceeding \(Y\), and \(b_{k_j}=0\) otherwise. On the other hand, for \(S_2\), by a similar argument, we use the orthogonality of characters to obtain 
\begin{align*}
    S_2 &\,= \sum_{\chi \in \mcg_q} \sum_{k_1,k_2,\dots,k_\ell \ge 1}b_{k_1}b_{k_2}\cdots b_{k_\ell}\chi(k_1)\chi(k_2)^2\cdots\chi(k_\ell)^\ell\sum_{m,n \in \mbn}r(m)r(n)\chi(m)\overline{\chi(n)} \\
    &\,= \sum_{k_1,k_2,\dots,k_\ell \ge 1}b_{k_1}b_{k_2}\cdots b_{k_\ell}\sum_{m,n \in \mbn}r(m)r(n)\sum_{\chi \in \mcg_q}\chi(k_1 k_2^2\cdots k_\ell^\ell m)\overline{\chi(n)} \\
    &\,= \phi(q)\sum_{k_1,k_2,\dots,k_\ell \ge 1}b_{k_1}b_{k_2}\cdots b_{k_\ell}\sum_{\substack{m,n \in \mbn \\ Km\equiv n\pmod{q}\\(n,q)=1}}r(m)r(n).
\end{align*}
Here, in the last step, we set \(K\coloneqq k_1 k_2^2\cdots k_\ell^\ell\). Noting that \(r(p)\) vanishes at primes \(p > X\), we define \(L(1,\chi;X)= \prod_{p \le X}\big(1-\chi(p)p\inv\big)\inv= \sum_{k \ge 1}a_k\chi(k)\). We observe that \(b_k \ge a_k\) for all \(k \ge 1\), since \(Y \ge X\). Thus, we have
\[
 S_2 \ge \phi(q)\sum_{k_1,k_2,\dots,k_\ell \ge 1}a_{k_1}a_{k_2}\cdots a_{k_\ell}\sum_{\substack{m,n \in \mbn \\ Km\equiv n\pmod{q}\\(n,q)=1}}r(m)r(n).
\]
Using that \(a_k\ge 0\) and that \(r(n)\) is completely multiplicative, we obtain
\begin{align}\label{S2lowerforthm1}
    S_2&\,\ge\phi(q)\sum_{k_1,k_2,\dots,k_\ell \ge 1}a_{k_1}a_{k_2}\cdots a_{k_\ell}\sum_{\substack{m,u \in \mbn: K \mid n \\ Km\equiv n\pmod{q}\\(n,q)=1}}r(m)r(n) \nonumber \\ 
    &\,=\phi(q)\sum_{k_1,k_2,\dots,k_\ell \ge 1}a_{k_1}a_{k_2}\cdots a_{k_\ell}r(K)\sum_{\substack{m,u \in \mbn \\ Km\equiv Ku\pmod{q}\\(u,q)=1}}r(m)r(u) \nonumber \\ 
    &\,=\phi(q)\sum_{k_1,k_2,\dots,k_\ell \ge 1}a_{k_1}a_{k_2}\cdots a_{k_\ell}r(K)\sum_{\substack{m,u \in \mbn \\ m\equiv u\pmod{q}\\(u,q)=1}}r(m)r(u).
\end{align}
The last step follows from the fact that \(q\) is prime. Consequently, we get the identity (see \cite[p. 10]{yang2023omega})
\[
\phi(q)\sum_{\substack{m,u \in \mbn \\ Km\equiv Ku\pmod{q}\\(u,q)=1}}r(m)r(u)=\phi(q)\sum_{\substack{m,u \in \mbn \\ m\equiv u\pmod{q}\\(u,q)=1}}r(m)r(u).
\]
Combining \eqref{S1equationforthm1} and \eqref{S2lowerforthm1}, we have 
\begin{align*}
    \frac{S_2}{S_1} &\,\ge \sum_{k_1,k_2,\dots,k_\ell \ge 1}a_{k_1}a_{k_2}\cdots a_{k_\ell}r(K) = \prod_{j=1}\sum_{k_j\ge1}a_{k_j}r(k_j)^j \\
    &\,=\prod_{j=1}^\ell\prod_{p \le X}\Big(1-\frac{r(p)^j}{p} \Big)\inv= \prod_{j=1}^\ell\prod_{p \le X} \Big(\frac{p}{p-1} \cdot \frac{p-1}{p-r(p)^j}\Big) \eqqcolon \prod_{j=1}^\ell \mcp_1\mcp_2.
\end{align*}
\par
We now estimate \(\mcp_1\) and \(\mcp_2\) respectively. For \(\mcp_1\), we apply Mertens’ theorem as given in \cite[Eq. (3.28)]{Rosser1962approximate}. More precisely,
\begin{align}\label{P1forthm1}
    \mcp_1 = \prod_{p \le X}\frac{p}{p-1} \ge \mme^\gamma\log X \Big(1-\frac{1}{2(\log X)^2}\Big).
\end{align}
For \(\mcp_2\), the bound \(r(p)^j = (1-p/X)^j\ge1-jp/X\) implies that
\[
P_2 = \prod_{j=1}^\ell \frac{p-1}{p-r(p)^j} \ge \exp\Big(\sum_{p \le X}\log \Big(1-\frac{jp}{jp+(p-1)X}\Big) \Big).
\]
By the prime number theorem, we have 
\[
\sum_{p \le X}\frac{p}{jp+X(p-1)} \le \sum_{p \le X}\Big(\frac{1}{X}+\frac{2}{pX}\Big) \le \frac{1}{\log X}+O\big((\log X)^{-3}\big).
\]
Therefore, combining this with the inequality \(\mme^{-x}\ge 1-x\), we obtain
\begin{align}\label{P2forthm1}
    \mcp_2 \ge \exp\Big(-\frac{j}{\log X}\big(1+O\big((\log X)^{-2}\big)\big)\Big) \ge 1-\frac{j}{\log X}\big(1+O\big((\log X)^{-2}\big)\big).
\end{align}
From \eqref{P1forthm1} and \eqref{P2forthm1}, we deduce that
\[
\mcp_1\mcp_2 \ge \mme^\gamma\log X\Big(1-\frac{j}{\log X}\big(1+O\big((\log X)^{-2}\big)\big)\Big)
\]
and
\begin{align}\label{S2S1forthm1}
    \frac{S_2}{S_1} \ge \mme^{\ell\gamma}(\log X)^\ell\prod_{j=1}^\ell\Big(1-\frac{j}{\log X}\big(1+O\big((\log X)^{-2}\big)\big)\Big).
\end{align}
A direct computation shows that the product on the right-hand side equals
\[
1-\frac{1}{\log X}\sum_{j=1}^\ell j +O_\ell\big((\log X)^{-2}\big) = 1 - \frac{(\ell+1)\ell}{2}\frac{1}{\log X}+O_\ell\big((\log X)^{-2}\big).
\]
Substituting this into \eqref{S2S1forthm1}, we have
\begin{align}\label{S2S1finalforthm1}
    \frac{S_2}{S_1} \ge \mme^{\ell\gamma}\Big((\log X)^\ell-\frac{(\ell+1)\ell}{2}(\log X)^{\ell-1} +O_\ell\big((\log X)^{\ell-2}\big)\Big).
\end{align}
\par
To apply the approximate formula \eqref{approximateforthm1}, it remains to show that the contributions of \(\chi_0\) and \(\chi_\mme\) are negligible. We consider the contribution of \(\chi_0\), while the case of \(\chi_\mme\) is analogous. According to \cite[pp. 839-841]{Aistleitner2019QJMath}, we have 
\[
|R(\chi_0)|^2 \le \exp\Big((1+o(1)\frac{2\log q}{\delta}\Big),\ \ L(1,\chi_0;Y) \ll (\log q)^{20}
\]
and 
\[
S_1 \ge \exp\Big((1+o(1)\Big(1+\frac{2-\log 4}{\delta}\Big)\log q\Big).
\]
Therefore, we require \(1+(2-\log 4)/\delta > 2/\delta\), that is, \(\delta > \log 4\). Recalling that \(X=\log q \log_2 q/\delta\), and taking \(\delta\) sufficiently close to \(\log4\) in \eqref{S2S1finalforthm1}, we obtain
\[
\frac{S_2}{S_1} \ge \mme^{\ell\gamma}\big\{(\log_2 q)^\ell + \ell(\log_3 q - C(\ell))(\log_2 q)^{\ell-1} + O_\ell\big((\log_2 q)^{\ell-2}(\log_3 q)^2\big) \big\},
\]
where \(C(\ell) = (\ell+1)/2 + \log_2 4\). Combining the approximate formula \eqref{approximateforthm1} with the above estimate for \(S_2/S_1\), we complete the proof of Theorem \ref{thm1}.

\section{Proof of Theorem \ref{thm2}}\label{sectionforthm2}
In this section, we continue to apply the long resonator method developed in \cite{Aistleitner2019QJMath}. Let \(X=\kappa\log q \log_2 q\), where \(\kappa>0\) will be chosen later. We define the resonator
\[
R(\chi) = \prod_{p \le X}(1-r(p)\chi(p))\inv = \sum_{n \in \mbn}r(n)\chi(n).
\]
Unlike in Section \ref{sectionforthm1}, the completely multiplicative function \(r(n)\) is defined by
\begin{align*}
    r(p) =
    \begin{cases}
        1-\big(\frac{p}{X}\big)^{\sigma}, ~ &\operatorname{if} p \le X, \\
		0, ~ &\operatorname{if} p > X.
    \end{cases}
\end{align*}
at primes \(p\). By the prime number theorem, we obtain 
\begin{align}\label{Rupperforthm2}
    |R(\chi)|^2 \le q^{2\kappa\sigma+o(1)}.
\end{align}
Set \(Y=(\log q)^{\frac{3}{\sigma-1/2}}\). Since \(\ell\) is a fixed, Lemma \ref{lemma3forthm2} yields the following approximate formula
\begin{align}\label{approximateforthm2}
    \prod_{j=1}^\ell L(\sigma,\chi^j) = \prod_{j=1}^\ell L(\sigma,\chi^j;Y)\big(1+O\big((\log q)^{-\frac{1}{4}}\big)\big), \ \ \forall \chi \in \mcg_q \setminus (E(q) \cup \{\chi_0\}).
\end{align}
Here, the cardinality of \(E(q)\) satisfies
\begin{align}\label{Equpperforthm2}
    \#E(q) \le \ell\cdot\#\mce(q) \ll_\ell q^{\frac{9/4-3\sigma/2}{7/4-\sigma/2}+o(1)}.
\end{align}
Noting that 
\begin{align}\label{expforthm2}
\prod_{j=1}^\ell |L(\sigma,\chi^j;Y)| = \exp\Big(\re\sum_{j=1}^\ell \sum_{p \le Y}\frac{\chi(p)^j}{p^\sigma} +O(1) \Big),   
\end{align}
we next estimate the sum in the exponent on the right-hand side.
\par
To this end, we define the following two quantities:
\[
S_1 = \sum_{\chi \in \mcg_q} |R(\chi)|^2 \ \text{and} \ \ S_2 = \sum_{\chi \in \mcg_q}\re\Big(\sum_{j=1}^\ell \sum_{p \le Y}\frac{\chi(p)^j}{p^\sigma} \Big) |R(\chi)|^2.
\]
For \(S_1\), expanding \(|R(\chi)|^2\) and using the orthogonality of characters, we get
\begin{align}\label{S1equationforthm2}
    S_1 = \sum_{\chi \in \mcg_q} \sum_{m,n \in \mbn}r(m)r(n)\chi(m)\overline{\chi(n)} = \phi(q)\sum_{\substack{m,n \in \mbn \\ m\equiv n\pmod{q}\\(n,q)=1}}r(m)r(n).
\end{align}
For \(S_2\), we again apply the orthogonality of characters to obtain
\begin{align*}
    S_2 &\,= \sum_{j=1}^\ell \sum_{p \le Y}\frac{1}{p^\sigma}\re\Big(\sum_{m,n \in \mbn}r(m)r(n)\sum_{\chi \in \mcg_q}\chi(p^jm)\overline{\chi(n)} \Big) \\
    &\,= \phi(q)\sum_{j=1}^\ell \sum_{p \le Y}\frac{1}{p^\sigma}\sum_{\substack{m,n\in \mbn \\ p^jm\equiv n \pmod{q}\\(n,q)=1 }}r(m)r(n).
\end{align*}
Since \(r(n)\) is non-negative and completely multiplicative, and \(Y \ge X\), we have
\[
S_2 \ge \phi(q)\sum_{j=1}^\ell \sum_{p \le X}\frac{r(p)^j}{p^\sigma}\sum_{\substack{m,u\in \mbn \\ p^jm\equiv p^ju \pmod{q}\\(u,q)=1 }}r(m)r(u).
\]
For all \(p \le X = \kappa\log q\log_2 q\) and all \(j \in \{1,\dots,\ell\}\), we have \((p^j,q)=1\). Thus, 
\[
\phi(q)\sum_{\substack{m,u\in \mbn \\ p^jm\equiv p^ju \pmod{q}\\(u,q)=1 }}r(m)r(u)= \phi(q)\sum_{\substack{m,u\in \mbn \\ m\equiv u \pmod{q}\\(u,q)=1 }}r(m)r(u).
\]
Together with the above equation and \eqref{S1equationforthm2}, this yields the lower bound
\begin{align}\label{S2S1forthm2}
\frac{S_2}{S_1} \ge \sum_{j=1}^\ell \sum_{p \le X}\frac{r(p)^j}{p^\sigma}.
\end{align}
By \cite[Eq. (4.10)]{qiyuyang2026joint}, we have
\[
\sum_{j=1}^\ell \sum_{p \le X} \frac{r(p)^j}{p^\sigma} \ge \big(S(\sigma,\ell) +o(1) \big) \frac{X^{1-\sigma}}{\log X},
\]
where
\begin{align*}
S(\sigma,\ell) & \, = \frac{\ell}{1-\sigma} - \frac{(\ell+1)\ell}{2} + \sum_{m=2}^\ell(-1)^m \binom{\ell +1}{m+1} \frac{1}{1+\sigma(m-1)} \\
   & \, = \frac{\ell}{1-\sigma} + \sum_{m=1}^\ell (-1)^m \binom{\ell+1}{m+1} \frac{1}{1+\sigma(m-1)}.
\end{align*}
Substituting this into \eqref{S2S1forthm2} and recalling the definition of \(X\), we obtain
\begin{align}\label{S2S1finalforthm2}
    \frac{S_2}{S_1} \ge \big(S(\sigma,\ell) +o(1) \big)\frac{\kappa^{1-\sigma}(\log q)^{1-\sigma}}{(\log_2 q)^\sigma}.
\end{align}
\par
Note that our approximate formula \eqref{approximateforthm2} may fail on \(E(q)\cup\{\chi_0\}\). It therefore remains to show that the contribution from this set is negligible. To this end, we require a lower bound for \(S_1\), which is given in \cite[p. 16]{yang2023omega} as
\begin{align}\label{S1lowerforthm2}
    S_1 \ge (q-1)\sum_{n\in\mbn}r(n)^2 \ge q^{1+\kappa\sigma(1-c(\sigma))+o(1)},
\end{align}
where
\[
c(\sigma) \coloneqq \int_0^1 \frac{\mathrm{d} t}{2t^{-\sigma}-1}.
\]
Combining \eqref{Rupperforthm2}, \eqref{Equpperforthm2} and \eqref{S1lowerforthm2}, we require that
\begin{align}\label{kappaforthm2}
2\kappa\sigma + \frac{9/4-3\sigma/2}{7/4-\sigma/2} < 1+\kappa\sigma(1-c(\sigma)).   
\end{align}
We choose \(\kappa\) such that \eqref{kappaforthm2} holds. This ensures that the total contribution of characters in \(E(q)\cup\{\chi_0\}\) is negligible. Then, by combining \eqref{approximateforthm2}, \eqref{expforthm2} and \eqref{S2S1finalforthm2}, we complete the proof of Theorem \ref{thm2}.

\section{Proof of Theorem \ref{thm3}}\label{sectionforthm3}
Following the approach of Sections \ref{sectionforthm1} and \ref{sectionforthm2}, we first establish an effective approximate formula. Set \(Y=\exp\big((3\log q)^2 \big)\). Then for some \(A>0\), we have
\begin{align}\label{approximateforthm3}
    (-1)^\ell \prod_{j=1}^\ell \frac{L^\prime}{L}\big(1,\chi^j \big) = \prod_{j=1}^\ell \sum_{n \le Y}\frac{\Lambda(n)\chi^j(n)}{n}+O\big(q^{-A}(\log Y)^{\ell-1}\big),\ \ \forall \chi \in \mcg_\ell(q).
\end{align}
In fact, Lemma \ref{lemma1forthm3} implies that for each \(j \in \{1,\dots,\ell\}\), we have
\[
    -\frac{L^\prime}{L}(1,\chi^j) = \sum_{n \le Y}\frac{\Lambda(n)\chi(n)^j}{n}+O(q^{-A}), \ \ \forall \chi^j \in \mcg_q^\ast.
\]
Let \(D_j(\chi)\) and \(\mce\) denote the main term and the error term, respectively. Since \(\ell\) is a fixed,
\[
(-1)^\ell \prod_{j=1}^\ell \frac{L^\prime}{L}\big(1,\chi^j \big) = \prod_{j=1}^\ell (D_j(\chi)+\mce) = \prod_{j=1}^\ell D_j(\chi) + \sum_{\emptyset \neq I \subset\{1,\dots,\ell\}}\Big(\prod_{j \not\in I}D_j(\chi)\Big)\Big(\prod_{j \in I}\mce \Big).
\]
We observe that, in the sum over the non-empty set \(I\) on the right-hand side, the contribution corresponding to \(|I|=1\) is largest. Noting that 
\[
|D_j(\chi)| \le \sum_{n \le Y}\frac{\Lambda(n)}{n} \ll \log Y,
\]
we see that \eqref{approximateforthm3} holds, since the total contribution of the error terms is bounded by \(q^{-A}(\log q)^{\ell-1}\). 
\par
Next, we study extreme values of \(\re\big(\prod_{j=1}^\ell D_j(\chi)\big)\). Set \(X= \tau\log q\log_2 q\), where \(\tau>0\) will be chosen later. We use the same resonator \(R(\chi) = \sum_{n \in \mbn}r(n)\chi(n)\) as in Section \ref{sectionforthm1}, and define
\[
S_1 = \sum_{\chi \in \mcg_q} |R(\chi)|^2,   \ \ S_2 = \sum_{\chi \in \mcg_q}\re\Big(\prod_{j=1}^\ell D_j(\chi)\Big)|R(\chi)|^2.
\]
For \(S_1\), by a similar argument, the orthogonality of characters gives
\begin{align}\label{S1equationforthm3}
    S_1 = \sum_{\chi \in \mcg_q} \sum_{m,n \in \mbn}r(m)r(n)\chi(m)\overline{\chi(n)} = \phi(q)\sum_{\substack{m,n \in \mbn \\ m\equiv n\pmod{q}\\(n,q)=1}}r(m)r(n).
\end{align}
By expanding \(|R(\chi)|^2\) and interchanging the order of summation, we obtain that \(S_2\) equals
\begin{align*}
    &\sum_{\chi \in \mcg_q} \re \Big(\sum_{n_1,n_2,\dots,n_\ell \le Y} \frac{\Lambda(n_1)\Lambda(n_2)\cdots\Lambda(n_\ell)}{n_1 n_2 \cdots n_\ell}\chi\big(n_1n_2^2\cdots n_\ell^\ell\big)\Big)\sum_{m,n\in\mbn}r(m)r(n)\chi(m)\overline{\chi(n)} \\
    &= \sum_{n_1,n_2,\dots,n_\ell \le Y} \frac{\Lambda(n_1)\Lambda(n_2)\cdots\Lambda(n_\ell)}{n_1 n_2 \cdots n_\ell} \re\Big(\sum_{m,n\in\mbn}r(m)r(n)\sum_{\chi \in \mcg_q}\chi\big(n_1n_2^2\cdots n_\ell^\ell m\big)\overline{\chi(n)} \Big).
\end{align*}
Using again the orthogonality of characters, we deduce that
\[
S_2 = \sum_{n_1,n_2,\dots,n_\ell\le Y} \frac{\Lambda(n_1)\Lambda(n_2)\cdots\Lambda(n_\ell)}{n_1 n_2 \cdots n_\ell}  \phi(q) \sum_{\substack{m,n\in\mbn \\ n_1n_2^2\cdots n_\ell^\ell m \equiv n \pmod{q}\\(n,q)=1}}r(m)r(n).
\]
Let \(N \coloneqq n_1n_2^2\cdots n_\ell^\ell\). Since \(r(n)\) is non-negative and completely multiplicative, and \(Y \ge X\), we have the following lower bound
\begin{align}\label{S2lowerforthm3}
    S_2 &\,\ge \phi(q)\sum_{n_1,n_2,\dots,n_\ell\le X} \frac{\Lambda(n_1)\Lambda(n_2)\cdots\Lambda(n_\ell)}{n_1 n_2 \cdots n_\ell}  \sum_{\substack{m,n\in\mbn:Nm \mid n \\ Nm \equiv n \pmod{q}\\(n,q)=1}}r(m)r(n) \nonumber \\
    &\, =\phi(q)\sum_{n_1,n_2,\dots,n_\ell\le X} \frac{\Lambda(n_1)\Lambda(n_2)\cdots\Lambda(n_\ell)}{n_1 n_2 \cdots n_\ell}r(N)\sum_{\substack{m,u\in\mbn \\ Nm \equiv Nu \pmod{q}\\(Nu,q)=1}}r(m)r(u) \nonumber \\
    &\, = \phi(q)\sum_{n_1,n_2,\dots,n_\ell\le X} \frac{\Lambda(n_1)\Lambda(n_2)\cdots\Lambda(n_\ell)}{n_1 n_2 \cdots n_\ell}r(N)\sum_{\substack{m,u\in\mbn \\ m \equiv u \pmod{q}\\(u,q)=1}}r(m)r(u).
\end{align}
In the last step, we use the fact that \(q\) is prime. Combining \eqref{S1equationforthm3} and \eqref{S2lowerforthm3}, we obtain 
\begin{align}\label{S2S1forthm3}
    \frac{S_2}{S_1} &\, \ge \sum_{n_1,n_2,\dots,n_\ell\le X} \frac{\Lambda(n_1)\Lambda(n_2)\cdots\Lambda(n_\ell)}{n_1 n_2 \cdots n_\ell}r(n_1)r(n_2)^2\cdots r(n_\ell)^\ell \nonumber \\
    &\,=\prod_{j=1}^\ell\sum_{n_j \le X}\frac{\Lambda(n_j)}{n_j}r(n_j)^j \ge \prod_{j=1}^\ell\sum_{p \le X}\frac{\log p}{p}r(p)^j.
\end{align}
\par
Define
\begin{align}\label{PjXforthm3}
    P_j(X) \coloneqq \sum_{p \le X}\frac{\log p}{p}r(p)^j = \sum_{p \le X}\frac{\log p}{p} \Big( 1-\frac{p}{X}\Big)^j.
\end{align}
Then by the binomial theorem, we obtain 
\begin{align*}
    P_j(X) &\,= \sum_{p \le X}\frac{\log p}{p} \sum_{m=0}^j (-1)^m\binom{j}{m}\Big(\frac{p}{X}\Big)^m = \sum_{m=0}^j (-1)^m\binom{j}{m} \frac{1}{X^m}\sum_{p \le X}p^{m-1}\log p \\
    &\,= \sum_{p \le X}\frac{\log p}{p} + \sum_{m=1}^j (-1)^m\binom{j}{m} \frac{1}{X^m}\sum_{p \le X}p^{m-1}\log p.
\end{align*}
For the first sum, the prime number theorem gives 
\begin{align*}
\sum_{p \le X}\frac{\log p}{p} &\,= \log X - \gamma -\sum_{k \ge 2}\sum_{p}\frac{\log p}{p^k} + O\big(\exp(-A\sqrt{\log X}) \big) \\
&\, =\log X - \gamma -\sum_{p}\frac{\log p }{p(p-1)} + O\big(\exp(-A\sqrt{\log X}) \big) 
\end{align*}
for some \(A >0\), see \cite[Eq. (2.31)]{Rosser1962approximate}. For \(m \ge 1\), 
\[
\frac{1}{X^m}\sum_{p \le X}p^{m-1}\log p = \frac{1}{m}+o(1).
\]
Thus, the remaining sum is given by 
\[
\sum_{m=1}^j (-1)^m\binom{j}{m} \frac{1}{m} +o(1) = \int_0^1 \frac{(1-x)^j-1}{x}\mathrm{d}x +o(1) = - \sum_{k=1}^j \frac{1}{k} +o(1).
\]
Let \(H_j\) denote the \(j\)-th harmonic series. Substituting the above into \eqref{PjXforthm3}, we obtain
\begin{align}\label{PjXfinalforthm3}
    P_j(X) = \log X - \gamma -\sum_{p}\frac{\log p }{p(p-1)} - H_j +o(1).
\end{align}
From \eqref{S2S1forthm3} and \eqref{PjXfinalforthm3}, it follows that
\[
    \frac{S_2}{S_1} \ge (\log X)^\ell+ (\ell(P^\ast+1)-(\ell+1))H_\ell)(\log X)^{\ell-1}+O_\ell\big((\log X)^{\ell-2}\big),
\]
where \(P^\ast = - \gamma -\sum_{p}(\log p)/(p(p-1))\).
Recalling that \(X=\tau\log q\log_2 q\), we have
\begin{align}\label{S2S1finalforthm3}
    \frac{S_2}{S_1} \ge &\,(\log_2 q)^\ell+\ell(\log_2 q)^{\ell-1}\log_3 q \nonumber \\
    &\,+(\ell\log\tau+\ell(P^\ast+1)-(\ell+1))H_\ell)(\log_2 q)^{\ell-1} +O_\ell\big((\log_2 q)^{\ell-2}(\log_3 q)^2\big).
\end{align}
\par
To control the contributions of \(\chi_0\) and \(\chi_\mme\), we use the following lower bound for \(S_1\) (see \cite[p. 10]{yang2023omega}):
\begin{align}\label{S1lowerforthm3}
    S_1 \ge (q-1)\sum_{n \in\mbn}r(n)^2 \ge q^{1+\tau(2-\log 4)+o(1)}.
\end{align}
Furthermore, 
\[
|R(\chi_0)|^2 \le q^{2\tau+o(1)}, \ \ \Big|\re\Big(\prod_{j=1}^\ell D_j(\chi_0)\Big)\Big| \ll_\ell \log Y \ll q^{o(1)}.
\]
Therefore,
\[
\Big|\re\Big(\prod_{j=1}^\ell D_j(\chi_0) |R(\chi_0)|^2 \Big)\Big| \le  q^{2\tau+o(1)}.
\]
Hence, it is enough to impose \(2\tau < 1+\tau(2-\log 4)\), that is, \(\tau < (\log 4)\inv\) so that the contribution of \(\chi_0\) can be neglected. The case of \(\chi_\mme\) is completely analogous. Choosing \(\tau\) sufficiently close to \((\log 4)\inv\) in \eqref{S2S1finalforthm3} and combining with the approximate formula \eqref{approximateforthm3},  we complete the proof of Theorem \ref{thm3}.

\section{Proof of Theorem \ref{thm4}}\label{sectionforthm4}
Set \(Y=(\log q)^\beta\), where \(\beta>1\) satisfies the conditions in Lemma \ref{lemma2forthm4}. By an argument similar to that in Section \ref{sectionforthm3}, Lemma \ref{lemma2forthm4} yields
\begin{align}\label{approximateforthm4}
(-1)^\ell\prod_{j=1}^\ell\frac{L^\prime}{L}(\sigma,\chi^j) =  \prod_{j=1}^\ell\sum_{n\le Y}\frac{\Lambda(n)\chi(n)^j}{n^\sigma}+O\big((\log q)^{-A}\big),\ \ \forall\chi \in \mcg_\ell(q)\setminus E^\prime(q).
\end{align}
Here, the set \(E^\prime(q)\) satisfies
\begin{align}\label{Equpperforthm4}
\#E^\prime(q) \ll q^{\frac{3(1-\sigma+\varepsilon)}{2-\sigma+\varepsilon}+o(1)}.
\end{align}
\par
Let \(X=\eta\log q\log_2 q\), where \(\eta\) will be chosen later. Since \(\beta >1\), we have \(Y \ge X\). We use the same resonator \(R(\chi)=\sum_{n\in\mbn}r(n)\chi(n)\) as in Section \ref{sectionforthm2}, and define
\[
S_1 = \sum_{\chi \in \mcg_q} |R(\chi)|^2,   \ \ S_2 = \sum_{\chi \in \mcg_q}\re\Big(\prod_{j=1}^\ell D_j(\sigma,\chi)\Big)|R(\chi)|^2,
\]
where 
\[
D_j(\sigma,\chi) = \sum_{n \le Y}\frac{\Lambda(n)\chi(n)^j}{n^\sigma}.
\]
\par
For \(S_1\), we have 
\begin{align}\label{S1equationforthm4}
    S_1 = \sum_{\chi \in \mcg_q} \sum_{m,n \in \mbn}r(m)r(n)\chi(m)\overline{\chi(n)} = \phi(q)\sum_{\substack{m,n \in \mbn \\ m\equiv n\pmod{q}\\(n,q)=1}}r(m)r(n).
\end{align}
On the other hand, by the same argument as in Section \ref{sectionforthm3}, expanding \(|R(\chi)|^2\) and interchanging the order of summation, we obtain that \(S_2\) equals
\begin{align*}
    &\sum_{\chi \in \mcg_q} \re \Big(\sum_{n_1,n_2,\dots,n_\ell \le Y} \frac{\Lambda(n_1)\Lambda(n_2)\cdots\Lambda(n_\ell)}{n_1^\sigma n_2^\sigma \cdots n_\ell^\sigma}\chi\big(n_1n_2^2\cdots n_\ell^\ell\big)\Big)\sum_{m,n\in\mbn}r(m)r(n)\chi(m)\overline{\chi(n)} \\
    &= \sum_{n_1,n_2,\dots,n_\ell \le Y} \frac{\Lambda(n_1)\Lambda(n_2)\cdots\Lambda(n_\ell)}{n_1^\sigma n_2^\sigma \cdots n_\ell^\sigma} \re\Big(\sum_{m,n\in\mbn}r(m)r(n)\sum_{\chi \in \mcg_q}\chi\big(n_1n_2^2\cdots n_\ell^\ell m\big)\overline{\chi(n)} \Big) \\
    &=\sum_{n_1,n_2,\dots,n_\ell\le Y} \frac{\Lambda(n_1)\Lambda(n_2)\cdots\Lambda(n_\ell)}{n_1^\sigma n_2^\sigma \cdots n_\ell^\sigma}  \phi(q) \sum_{\substack{m,n\in\mbn \\ N m \equiv n \pmod{q}\\(n,q)=1}}r(m)r(n).
\end{align*}
In the last step, we use the orthogonality of characters and set \(N \coloneqq n_1n_2^2\cdots n_\ell^\ell \). Since \(r(n)\) is non-negative and completely multiplicative, and \(Y \ge X\), we have
\begin{align*}
    S_2 &\, \ge \phi(q) \sum_{n_1,n_2,\dots,n_\ell\le X} \frac{\Lambda(n_1)\Lambda(n_2)\cdots\Lambda(n_\ell)}{n_1^\sigma n_2^\sigma \cdots n_\ell^\sigma}  \sum_{\substack{m,n\in\mbn \\ N m \equiv n \pmod{q}\\(n,q)=1}}r(m)r(n) \\
    &\,\ge \phi(q)\sum_{n_1,n_2,\dots,n_\ell\le X} \frac{\Lambda(n_1)\Lambda(n_2)\cdots\Lambda(n_\ell)}{n_1^\sigma n_2^\sigma \cdots n_\ell^\sigma}  \sum_{\substack{m,n\in\mbn:Nm \mid n \\ Nm \equiv n \pmod{q}\\(n,q)=1}}r(m)r(n) \\
    &\, =\phi(q)\sum_{n_1,n_2,\dots,n_\ell\le X} \frac{\Lambda(n_1)\Lambda(n_2)\cdots\Lambda(n_\ell)}{n_1^\sigma n_2^\sigma \cdots n_\ell^\sigma}r(N)\sum_{\substack{m,u\in\mbn \\ Nm \equiv Nu \pmod{q}\\(Nu,q)=1}}r(m)r(u).
\end{align*}
Since \(q\) is prime, we have the following lower bound 
\begin{align}\label{S2lowerforthm4}
        S_2 \ge \phi(q)\sum_{n_1,n_2,\dots,n_\ell\le X} \frac{\Lambda(n_1)\Lambda(n_2)\cdots\Lambda(n_\ell)}{n_1^\sigma n_2^\sigma \cdots n_\ell^\sigma}r(N)\sum_{\substack{m,u\in\mbn \\ m \equiv u \pmod{q}\\(u,q)=1}}r(m)r(u).
\end{align}
Combining \eqref{S1equationforthm4} and \eqref{S2lowerforthm4}, we have
\begin{align}\label{S2S1forthm4}
    \frac{S_2}{S_1} \ge \sum_{n_1,n_2,\dots,n_\ell\le X} \frac{\Lambda(n_1)\Lambda(n_2)\cdots\Lambda(n_\ell)}{n_1^\sigma n_2^\sigma \cdots n_\ell^\sigma}r(N) \ge \prod_{j=1}^\ell\sum_{p \le X}\frac{\log p}{p^\sigma}r(p)^j.
\end{align}
\par
Define
\begin{align}\label{PjsigmaXforthm4}
    P_j(\sigma,X) \coloneqq \sum_{p \le X}\frac{\log p}{p^\sigma}r(p)^j = \sum_{p \le X}\frac{\log p}{p^\sigma} \Big( 1-\Big(\frac{p}{X}\Big)^\sigma\Big)^j.
\end{align}
By the binomial theorem, we obtain
\begin{align*}
    P_j(\sigma,X) &\,= \sum_{k=0}^j (-1)^k \binom{j}{k}X^{-k\sigma}\sum_{p \le X} p^{(k-1)\sigma}\log p \\
    &\,= \sum_{p \le X}\frac{\log p}{p^\sigma} + \sum_{k=1}^j (-1)^k \binom{j}{k}X^{-k\sigma}\sum_{p \le X} p^{(k-1)\sigma}\log p.
\end{align*}
For the first sum, the prime number theorem implies
\[
\sum_{p \le X}\frac{\log p}{p^\sigma} = (1+o(1))\frac{X^{1-\sigma}}{1-\sigma}.
\]
For the terms with \(k = 1,\dots,j\), by combining the prime number theorem with partial summation, we obtain
\[
 \sum_{k=1}^j (-1)^k \binom{j}{k}X^{-k\sigma}\sum_{p \le X} p^{(k-1)\sigma}\log p = (1+o(1)) \sum_{k=1}^j (-1)^k \binom{j}{k}X^{1-\sigma}\frac{1}{1+(k-1)\sigma}.
\]
Returning to \eqref{PjsigmaXforthm4}, we have
\[
P_j(\sigma,X) =(1+o(1)) X^{1-\sigma}\sum_{k=0}^j (-1)^k \binom{j}{k} \frac{1}{1+(k-1)\sigma}.
\]
We now focus on the sum on the right-hand side above. Making the change of variable \(\alpha = (1-\sigma)/\sigma\), we obtain that the sum is given by
\begin{align*}
    &\frac{1}{\sigma}\sum_{k=0}^j (-1)^k \binom{j}{k}\frac{1}{k+\alpha} = \frac{1}{\sigma}\sum_{k=0}^j (-1)^k \binom{j}{k}\int_0^1 t^{k+\alpha-1}\mathrm{d}t \\
    &=\frac{1}{\sigma}\int_0^1 t^{\alpha-1}\sum_{k=0}^j (-1)^k \binom{j}{k} t^k\mathrm{d}t = \frac{1}{\sigma} B(\alpha,j+1) = \frac{1}{\sigma}\frac{\Gamma(\alpha)\Gamma(j+1)}{\Gamma(\alpha+j+1)}.
\end{align*}
Here, \(B(m,n)\) and \(\Gamma(n)\) denote the Beta function and the Gamma function, respectively. Using the identities
\[
\frac{1}{\sigma}\Gamma\Big(\frac{1-\sigma}{\sigma}\Big)=\frac{1}{1-\sigma}\Gamma\Big(\frac{1}{\sigma}\Big), \ \Gamma(j+1)=j! \ \text{and} \ \Gamma\Big(j+\frac{1}{\sigma}\Big) = \Gamma\Big(\frac{1}{\sigma}\Big)\prod_{m=0}^{j-1}\Big(m+\frac{1}{\sigma}\Big),
\]
we obtain
\begin{align}\label{PjsigmaXfinalforthm4}
    P_j(\sigma,X) = (1+o(1)) \frac{X^{1-\sigma}}{1-\sigma} j! \prod_{m=0}^{j-1}\Big(m+\frac{1}{\sigma}\Big)\inv.
\end{align}
Consequently, \eqref{S2S1forthm4} and \eqref{PjsigmaXfinalforthm4} give that
\[
\frac{S_2}{S_1} \ge (1+o(1)) \Big(\frac{X^{1-\sigma}}{1-\sigma} \Big)^\ell \prod_{j=1}^\ell \Big(j! \prod_{m=0}^{j-1}\bigg(m+\frac{1}{\sigma}\Big)\inv\bigg).
\]
Recalling that \(X=\eta\log q \log_2 q\), we have
\begin{align}\label{S2S1finalforthm4}
    \frac{S_2}{S_1} \ge
\big(\eta^{\ell(1-\sigma)}H(\sigma,\ell)+o(1) \big)(\log q)^{\ell(1-\sigma)}(\log_2 q)^{\ell(1-\sigma)},
\end{align}
where
\[
H(\sigma,\ell) = \prod_{j=1}^\ell \bigg(\frac{j!}{1-\sigma}\prod_{m=0}^{j-1} \Big(m+\frac{1}{\sigma}\Big)\inv\bigg).
\]
\par
It remains to control the contribution of those characters for which the approximate formula \eqref{approximateforthm4} does not hold. To this end, we use the following lower bound for \(S_1\) (see \cite[p. 16]{yang2023omega}):
\begin{align}\label{S1lowerforthm4}
    S_1 \ge (q-1)\sum_{n\in\mbn}r(n)^2 \ge q^{1+\eta\sigma(1-c(\sigma))+o(1)},
\end{align}
where
\[
c(\sigma) \coloneqq \int_0^1 \frac{\mathrm{d}t}{2t^{-\sigma}-1}.
\]
Furthermore, we have
\[
|R(\chi)|^2 \le q^{2\eta+o(1)}, \ \ |D_j(\sigma,\chi)| \ll \frac{Y^{(1-\sigma)}}{1-\sigma} \ll q^{o(1)}.
\]
Thus, for all characters \(\chi \in \mcg_q\),
\[
\Big|\re\Big(\prod_{j=1}^\ell D_j(\sigma,\chi)|R(\chi)|^2\Big)\Big| \le q^{2\eta+o(1)}.
\]
By combining this with \eqref{Equpperforthm4} and \eqref{S1lowerforthm4}, we require that \(\eta\) satisfies 
\[
2\eta\sigma+ \frac{3(1-\sigma+\varepsilon)}{2-\sigma+\varepsilon} < 1+ \eta\sigma(1-c(\sigma)).
\]
Choosing \(\eta\) to satisfy the above condition and combining \eqref{approximateforthm4} and \eqref{S2S1finalforthm4}, we complete the proof of Theorem \ref{thm4}.

\section*{Acknowledgments}
The author is grateful to the referee for constructive comments and suggestions that greatly improved the quality of the paper. The author would like to thank Prof. Zhonghua Li for his long-standing encouragement. The author also thank Dr. Zikang Dong, Dr. Yutong Song and Dr. Qiyu Yang for their helpful comments.

	\bibliographystyle{siam}
    \bibliography{reference}

@article{aistleitner2016MathAnn,
  author = {C. Aistleitner},
  title = {Lower bounds for the maximum of the {R}iemann zeta function along vertical lines},
  journal = {Math. Ann.},
  volume = {365},
  number = {1-2},
  pages = {473--496},
  year = {2016}
}

@article{aistleitner2019IMRN,
  title={Extreme values of the {R}iemann zeta function on the 1-line},
  author={Aistleitner, Christoph and Mahatab, Kamalakshya and Munsch, Marc},
  journal={Int. Math. Res. Not.},
  volume={IMRN 2019},
  number={22},
  pages={6924--6932},
  year={2019}
}

@article{Aistleitner2019QJMath,
  author = {C. Aistleitner and K. Mahatab and M. Munsch and A. Peyrot},
  title = {On large values of \({L}(\sigma, \chi)\)},
  journal = {Q. J. Math.},
  volume = {70},
  number = {3},
  pages = {831--848},
  year = {2019}
}

@article{bondarenko2017Duke,
  author = {A. Bondarenko and K. Seip},
  title = {Large greatest common divisor sums and extreme values of the {R}iemann zeta function},
  journal = {Duke Math. J.},
  volume = {166},
  number = {9},
  pages = {1685--1701},
  year = {2017}
}

@article{bondarenko2018MathAnn,
    author = {Bondarenko, A and Seip, K},
    title = {Extreme values of the {R}iemann zeta function and its argument},
    journal = {Math. Ann.},
    volume = {372},
    number = {3-4},
    pages = {999--1015},
    year = {2018}
}

@article{bondarenko2023dichotomy,
  title={A dichotomy for extreme values of zeta and {D}irichlet {$L$}-functions},
  author={A. Bondarenko and P. Darbar and M. V. Hagen and W. Heap and K. Seip},
  journal={Bull. Lond. Math. Soc.},
  volume={55},
  number={6},
  pages={2963--2975},
  year={2023}
}

@article{chirre2019extreme,
  title={Extreme values for ${S}_n (\sigma, t)$ near the critical line},
  author={Chirre, Andr{\'e}s},
  journal={J. Number Theory},
  volume={200},
  pages={329--352},
  year={2019}
}

@article{dong2023Onde,
  title = {On derivatives of zeta and {$L$}-functions},
  author = {Zikang Dong and Yutong Song and Weijia Wang and Hao Zhang},
  journal = {Ramanujan J.},
  volume = {66},
  number = {1},
  pages = {5--21},
  year={2025}
}

@article{Granville2001JAMS,
    author = {A. Granville and K. Soundararajan},
    title = {Large character sums},
    journal = {J. Amer. Math. Soc.},
    year = {2001},
    volume ={14},
    pages ={365--397}
}

@article{Granville2003Geom,
    author = {Andrew Granville and Kannan Soundararajan},
    title = {The distribution of values of ${L}(1,\chi_D)$},
    journal = {Geom. Funct. Anal.},
    year = {2003},
    volume = {13},
    number = {5},
    pages = {992--1028}
}

@incollection{Granville2006Ramanujan,
  title = {Extreme values of $|\zeta(1 + it)|$},
  booktitle = {The {R}iemann zeta function and related themes: {P}apers in {H}onour of {P}rofessor {K}. {R}amachandra},
  series = {Ramanujan Math. Soc. Lect. Notes Ser.},
  volume = {2},
  author = {A. Granville and K. Soundararajan},
  publisher = {Ramanujan Math. Soc.},
  address = {Mysore},
  year = {2006},
  pages = {65–-80}
}

@article{heath1992zerofree,
  title = {Zero-free regions for Dirichlet ${L}$-functions, and the least prime in an arithmetic progression},
  author = {Heath-Brown, D. R.},
  year = {1992},
  volume = {64},
  number = {3},
  journal = {Proc. Lond. Math. Soc.},
  pages = {265--338}
}

@article{lamzouri2011IMRN,
    author = {Y. Lamzouri},
    title = {On the distribution of extreme values of zeta and ${L}$-functions in the strip \(1/2 < \sigma <1\)},
    journal = {Int. Math. Res. Not. IMRN},
    year = {2011},
    volume = {2011},
    number = {139},
    pages = {5449--5503}
}

@article{levinson1972harmonic,
  title={{$\Omega$}-Theorems for {R}iemann's Zeta-Function at Harmonic Combinations of Points},
  author={Levinson, Norman},
  journal={Proc. Nat. Acad. Sci.},
  volume={69},
  number={7},
  pages={1657--1658},
  year={1972}
}

@book{Montgomery1971,
  author = {H. L. Montgomery},
  title = {Topics in multiplicative number theory},
  series = {Lecture Notes in Mathematics},
  volume = {227},
  publisher = {Springer-Verlag},
  address = {Berlin-New York},
  year = {1971}
}

@article{Montgomery1977,
  author = {Hough Lowell Montgomery},
  title = {Extreme values of the {R}iemann zeta function},
  journal = {Comment. Math. Helv.},
  volume = {52},
  number = {4},
  pages = {511–-518},
  year = {1977}
}

@article{Mourtada2013IJNT,
   author = {M. R. Mourtada and V. K. Murty},
  title = {Omega theorems for ${L}^\prime(1,\chi_{D})$},
  journal = {Int. J. Number Theory},
  volume = {9},
  number = {3},
  pages = {561–-581},
  year = {2013}
}

@article{Rosser1962approximate,
    author = {J. B. Rosser and L. Schoenfeld},
    title = {Approximate formulas for some functions of prime numbers},
    journal = {Illinois J. Math.},
    year = {1962},
    volume = {6},
    pages ={64--94}
}

@article{soundararajan2008extreme,
  author = {K. Soundararajan},
  title = {Extreme values of zeta and {L}-functions},
  journal = {Math. Ann.},
  volume = {342},
  number = {2},
  pages = {467--486},
  year = {2008}
}

@article{voronin1988lower,
  author = {S. M. Voronin},
  title = {Lower bounds in {R}iemann zeta-function theory},
  journal = {Izv. Akad. Nauk SSSR Ser. Mat.},
  volume = {52},
  number = {4},
  pages = {882--892, 896},
  year = {1988}
}

@article{xumax2024JNT,
    author = {Xu M. W. and Yand D.},
    title = {Extreme values of {D}irichlet polynomials with multiplicative coefficients},
    journal = {J. Number Theory},
    volume = {258},
    pages = {173-180},
    year = {2024}
}

@article{XiaoYang2022BAMS,
  author = {X. Xiao and Q. Yang},
  title = {A note on large values of \({L}(\sigma, \chi)\)},
  journal = {Bull. Aust. Math. Soc.},
  volume = {105},
  number = {3},
  pages = {412–-418},
  year = {2022}
}

@article{yang2023BLMS,
  title={Extreme values of derivatives of zeta and {$L$}-functions},
  author = {Daodao Yang},
  journal={Bull. Lond. Math. Soc.},
  volume={56},
  number={1},
  pages={79--95},
  year={2023}
}

@article{yang2023omega,
  title = {Omega theorems for logarithmic derivatives of zeta and ${L}$-functions},
  author = {D. Yang},
  journal = {Preprint, arXiv:2311.16371},
  year = {2023}
}

@article{qiyuyang2026joint,
    author = {Q. Yang and S. Zhao},
    title = {Joint extreme values of the {R}iemann zeta function at harmonic points},
    journal = {Preprint, arXiv:2601.02623},
    year ={2026} 
}

@article{qiyu2024JNT,
    title = {Large values of $\zeta(s)$ for $1/2<${R}e$(s)<1$},
journal = {J. Number Theory},
volume = {254},
pages = {199-213},
year = {2024},
author ={Qiyu Yang}
}
\end{document}